\documentclass[reqno]{amsart}
\usepackage{amsfonts}
\usepackage{amssymb}
\usepackage{mathrsfs}
\usepackage{times}
\usepackage{xcolor}
\usepackage[square,numbers,sort&compress]{natbib}

\newcommand{\black}{\color{black}}

{\theoremstyle{plain}%
  \newtheorem{theorem}{Theorem}[section]
  \newtheorem{corollary}[theorem]{Corollary}
  \newtheorem{proposition}[theorem]{Proposition}
  \newtheorem{lemma}[theorem]{Lemma}
  \newtheorem{example}[theorem]{Example}
  
  \newtheorem{definition}[theorem]{Definition}

\newtheorem{remark}[theorem]{Remark}

\newfont{\hueca}{msbm10}
\def\hu #1{\hbox{\hueca #1}}\def\hu #1{\hbox{\hueca #1}}

\begin{document}

\title{Split Lie-Rinehart algebras}

\thanks{The first, second and fourth authors acknowledge financial assistance by the Centre for Mathematics of the University of Coimbra -- UID/MAT/00324/2013, funded by the Portuguese Government through FCT/MEC and co-funded by the European Regional Development Fund through the Partnership Agreement PT2020. Third and fourth authors are supported by the PCI of the UCA `Teor\'\i a de Lie y Teor\'\i a de Espacios de Banach', by the PAI with project numbers FQM298, FQM7156 and by the project of the Spanish Ministerio de Educaci\'on y Ciencia  MTM2013-41208P. The fourth author acknowledges the Funda\c{c}\~{a}o para a Ci\^{e}ncia e a Tecnologia for the grant with reference SFRH/BPD/101675/2014.}

\author[H. Albuquerque ]{Helena~Albuquerque}

\address{Helena~Albuquerque, CMUC, Departamento de Matem\'atica, Universidade de Coimbra, Apartado 3008,
3001-454 Coimbra, Portugal. \hspace{0.1cm} {\em E-mail address}: {\tt lena@mat.uc.pt}}

\author[E. Barreiro]{Elisabete~Barreiro}

\address{Elisabete~Barreiro, CMUC, Departamento de Matem\'atica, Universidade de Coimbra, Apartado 3008,
3001-454 Coimbra, Portugal. \hspace{0.1cm} {\em E-mail address}: {\tt mefb@mat.uc.pt}}{}

\author[A.J. Calder\'on]{A.J.~Calder\'on}

\address{A.J.~Calder\'on, Departamento de Matem\'aticas, Universidad de C\'adiz, Campus de Puerto Real, 11510, Puerto Real, C\'adiz, Espa\~na. \hspace{0.1cm} {\em E-mail address}: {\tt ajesus.calderon@uca.es}}{}

\author[Jos\'{e} M. S\'{a}nchez-Delgado]{Jos\'{e} M. S\'{a}nchez-Delgado}

\address{Jos\'{e} M. S\'{a}nchez-Delgado, CMUC, Departamento de Matem\'atica, Universidade de Coimbra, Apartado 3008, 3001-454 Coimbra, Portugal. \hspace{0.1cm} {\em E-mail address}: {\tt txema.sanchez@mat.uc.pt}}

\begin{abstract}
We introduce the class of split Lie-Rinehart algebras as the natural extension of the one of split Lie algebras. We show that if $L$ is a tight split Lie-Rinehart algebra over an associative and commutative algebra $A,$ then $L$ and $A$ decompose as the orthogonal direct sums $L = \bigoplus_{i \in I}L_i$, $A = \bigoplus_{j \in J}A_j$, where  any $L_i$ is a nonzero ideal of $L$, any  $A_j$ is a nonzero ideal of $A$, and both decompositions satisfy   that for any $i \in I$ there exists a unique $\tilde{i} \in J$ such that $A_{\tilde{i}}L_i \neq 0$. Furthermore any   $L_i$ is a split Lie-Rinehart algebra over $A_{\tilde{i}}$.
%As a consequence, $L$ is the external direct sum of the family of Lie-Rinehart algebras $\{(L_i,A_{\tilde{i}})\}_{i \in I}$.
 Also,   under mild conditions, it is shown that the above decompositions of $L$ and $A$ are  by means of the family of their, respective, simple ideals.

\medskip



\medskip

%{\centerline {\it Dedicated to  Professor Hans Schneider}}
\end{abstract}

\maketitle

%%%%%%%%%%%%%%%%%%%%%%%%%%%%%%%%%%%%%%%%%%%%%%%%%5
%%%%%%%%%%%%%%%%%%%%%%%%%%%%%%%%%%%%%%%%%%%%%%%%%5
\section{Introduction and first definitions}
%%%%%%%%%%%%%%%%%%%%%%%%%%%%%%%%%%%%%%%%%%%%%%%%%5
%%%%%%%%%%%%%%%%%%%%%%%%%%%%%%%%%%%%%%%%%%%%%%%%%5

Lie-Rinehart algebras were introduced by J. Herz in \cite{Herz}, being their theory mainly developed by R. Palais \cite{Palais} and G. Rinehart \cite{Rinehart}. A Lie-Rinehart algebra can be thought as  a Lie $\mathbb K$-algebra, which is simultaneosly an $A$-module, where $A$ is an associative and commutative $\mathbb K$-algebra, in such a way that both structures are related in an appropriate way. We cand find in \cite{Huebschmann,Huebschmann2,Huebschmann3} a first approach to this class of algebras. In the last years, Lie-Rinehart algebras have been considered in many areas of Mathematics, particulary from a geometric viewpoint (see for instance \cite{Mackenzie}) and of course from an algebraic viewpoint \cite{Recent1, Recent2, Recent3}. Some generalizations of Lie-Rinehart algebras, such as Lie-Rinehart superalgebras \cite{Chemla} or restricted Lie-Rinehart algebras \cite{Dokas}, have been recently studied.

On the other hand, we recall that the class of split Lie  algebras is specially related to addition quantum numbers, graded contractions and deformations. For instance, for a phy\-sical system  $L$, it is interesting to know in detail the structure of the split decomposition because its roots can be seen as certain eigenvalues which are the additive quantum numbers characterizing the state of such system. We note that  determining   the structure of different types of split algebras are becoming more meaningful in the area of research of mathema\-tical physics. In fact,  the structure of different classes of split algebras have been recently studied by using techniques of connections of roots (see for instance \cite{HomLie,YoLie,Yopoisson,tri,AMSi,Cao8,Cao3}).

In the present paper we introduce the class of split Lie-Rinehart algebras $(L,A)$ as the natural extension of the one of split Lie algebras, and study its structure. Our techniques consist in considering  the roots system of  $L$ as well as the weights system of $A$. In these two sets we introduce two  different  notions of  connections, the first one  among the elements in  the roots system of  $L$ and the second one among the elements in  the weights system of  $A$. Later, we relate both concepts to get our main results. We show that if $L$ is a tight split Lie-Rinehart algebra (with restrictions neither its dimension nor its base field) over an associative and commutative algebra $A,$ then $L$ and $A$ decompose as the direct sums  $$L = \bigoplus\limits_{i \in I}L_i, \hspace{0.5cm} A = \bigoplus\limits_{j \in J}A_j,$$ where any $L_i$ is a nonzero ideal of $L$ satisfying $[L_i,L_k]=0$ when $i \neq k$, and any  $A_j$ is a nonzero ideal of $A$ such that $A_jA_l=0$ when $j \neq l$. Moreover,  both decompositions satisfy that for any $i \in I$ there exists a unique $\tilde{i} \in J$ such that $$A_{\tilde{i}}L_i \neq 0.$$ Furthermore any   $L_i$ is a split Lie-Rinehart algebra over $A_{\tilde{i}}$.
   %As a consequence, $(L,A)$ is the external direct sum of the family of Lie-Rinehart algebras $\{(L_i,A_{\tilde{i}})\}_{i \in I}$.
Also, under mild conditions, it is shown that the above decompositions of $L$ and $A$ are  by means of the family of their, respective, simple ideals.

\medskip

Our paper is organized as follows. In Section 2 we develop    connection techniques in the framework of Lie-Rinehart algebras $(L,A)$ and  apply, as a first step,  all of these techniques to the study of the inner  structure of $L$. In Section 3  we get, as a second step,  a  decomposition of $A$ as direct sum of adequate ideals. In Section 4 we relate the obtained results  on $L$ and $A$, getting in Sections 2 and 3,   to prove our above mentioned  main results. Section 5 is devoted to show that, under mild conditions, the given decompositions of $L$ and $A$ are by means of the family of their, corresponding, simple ideals.

\begin{definition}\rm
Let ${\mathbb K}$ be an arbitrary  base field and $A$ a commutative and associative algebra over ${\mathbb K}$. A {\it derivation} on $A$ is a ${\mathbb K}$-linear map $D : A \to A$ which satisfies
\begin{equation}
D(ab) = D(a)b + aD(b) \hspace{0.2cm} \mbox{\it (Leibniz's law)}
\end{equation}
for all $a,b \in A$.
The set $\mbox{Der}(A)$ of all derivations of $A$ is a Lie ${\mathbb K}$-algebra with Lie bracket $[D,D'] = DD' - D'D$, and an $A$-module simultaneosly. These two structures are related by the following identity $$[D,aD'] = a[D,D'] + D(a)D', \hspace{0.2cm} \mbox{for all} \hspace{0.1cm} D,D' \in  \mbox{Der}(A).$$
\end{definition}

\begin{definition}\rm
A {\it Lie-Rinehart algebra} over an (associative and  commutative) ${\mathbb K}$-algebra $A$ is a Lie ${\mathbb K}$-algebra $L$ endowed with an $A$-module structure and with a map (usually called {\it anchor}) $$\rho : L \to \mbox{Der}(A),$$ which is simultaneously an $A$-module and a Lie algebras homomorphism, and  such that the following relation holds
\begin{equation}\label{fundamental}
[v,aw] = a[v,w] + \rho(v)(a)w,
\end{equation}
for any  $ v,w \in L$ and $a \in A.$ We denote it by $(L,A)$ or just by $L$ if there is not any possible confusion.
\end{definition}

\begin{example}\label{exa1}\rm
Any  Lie algebra $L$ is a Lie-Rinehart algebra over   $A := \mathbb{K}$ as consequence of  $\mbox{Der}({\mathbb K}) = 0$.
\end{example}

\begin{example}\label{exa2}\rm
 Any associative and commutative $\mathbb K$-algebra $A$ gives rise to  a Lie-Rinehart algebra by taking  $L := \mbox{Der}(A)$ and $\rho := Id_{\mbox{Der}(A)}$.
\end{example}

Throughout this paper $(L,A)$ is a Lie-Rinehart algebra with  restrictions neither on the dimension of $L$, nor  on the   dimension of $A$, nor on the base field ${\hu K}$. A {\it subalgebra} $(S,A)$ of $(L,A),$ $S$ for short, is a Lie subalgebra of $L$ such that  $AS \subset S$ and satisfying that $S$ acts on $A$ via the composition $$ S\hookrightarrow L \overset{\rho} \to {\mbox{Der}(A)}.$$

A subalgebra $(I,A)$, $I$ for short, of $L$ is called an {\it ideal} if $I$ is a Lie ideal of $L$ and satisfies
\begin{equation}\label{cond_ideal}
\rho(I)(A)L \subset I.
\end{equation}
As example of an  ideal we have ${\rm Ker}\rho$, the kernel of $\rho$.
\medskip

%We say a Lie-Rinehart algebra $L$ is {\it abelian} if $[L, L] = 0.$

We say that a Lie-Rinehart algebra $(L,A)$ is {\it simple} if $[L,L] \neq 0$, $AA \neq 0$, $AL \neq 0$ and its only ideals are   $\{0\}$, $L$  and ${\rm Ker}\rho$.

\medskip

Let us introduce the class of split algebras in the framework of Lie-Rinehart algebras. We begin  by recalling  the definition of a split Lie algebra.

\begin{definition}\rm
A {\it splitting Cartan subalgebra} $H$ of a Lie algebra $L$ is defined as a maximal abelian subalgebra (MASA) of $L$ satisfying that the adjoint mappings $ad(h),$ for $h \in H,$ are simultaneously diagonalizable. If $L$ contains a splitting Cartan subalgebra $H$ then $L$ is called a {\it split Lie algebra}.
\end{definition}

Meaning that we have a decomposition of the Lie algebra $L$ as $$L = H \oplus (\bigoplus\limits_{\gamma \in \Gamma}L_{\gamma}),$$ where $$L_{\gamma} := \{v_{\gamma} \in L : [h,v_{\gamma}] = \gamma(h)v_{\gamma} \hspace{0.2cm} {\rm for} \hspace{0.1cm} {\rm any} \hspace{0.1cm} h \in H\},$$ for a linear functional $\gamma : H \to \hu{K},$ and where $\Gamma := \{\gamma \in H^* \setminus \{0\} : L_{\gamma} \neq 0\}$ denotes  the corresponding {\it roots system}. The linear subspace $L_{\gamma}$, for $\gamma \in \Gamma$, is called {\it root space of} $L$ {\it associated to} $\gamma,$ the elements $\gamma \in \Gamma \cup \{0\}$ are called {\it roots} of $L.$

\begin{definition}\label{split}\rm
A {\it split Lie-Rinehart algebra} (with respect to a MASA $H$ of the Lie algebra $L$) is a Lie-Rinehart algebra  $(L,A)$  in which the Lie algebra $L$ contains a splitting Cartan subalgebra $H$ and the algebra $A$ is a weight module (with respect to $H$) in the sense that  $A$ decomposes  as $$A = A_0 \oplus (\bigoplus\limits_{\alpha \in \Lambda}A_{\alpha}),$$ where $$A_{\alpha} := \{a_{\alpha} \in A : \rho(h)(a_{\alpha}) = \alpha(h)a_{\alpha}, \hspace{0.1cm} {\rm for} \hspace{0.1cm} {\rm any} \hspace{0.1cm}  h \in H\},$$ for a linear functional $\alpha : H \to \mathbb{K}$, and where $\Lambda := \{\alpha \in H^* \setminus \{0\} : A_{\alpha} \neq 0\}$ denotes  the  {\it weights system} of $A$. The linear subspace $A_{\alpha},$ for $\alpha \in \Lambda,$ is called the {\it weight space} of $A$ {\it associate to} $\alpha,$ the elements $\alpha \in \Lambda \cup \{0\}$ are called {\it weights} of $A.$
\end{definition}

Taking into account Example \ref{exa1},  split Lie algebras are  examples  of  split Lie-Rinehart algebras.
 %we have the . Indeed, as it was shown in Example \ref{exa1} in this case $A = \mathbb{K}$ so ${\rm Der}(A) = 0$ and the decomposition of $A$ is trivial.
The present paper extends the structure theorems getting in \cite{YoLie} for split Lie algebras to the class of split Lie-Rinehart algebras.
 %We would like to know that most of the constructions carried out along this paper strongly involve the structure map $\rho$, which makes the proofs different.

 \medskip

From now on, $ (L,A)$ denotes a split Lie-Rinehart algebra (with respect to a MASA $H$ of $L$) being
$$\hbox{$L = L_0 \oplus (\bigoplus\limits_{\gamma \in \Gamma}L_{\gamma})$ and  $A = A_0 \oplus (\bigoplus\limits_{\alpha \in \Lambda}A_{\alpha})$}$$  the corresponding root and weight spaces decompositions.

\begin{lemma}\label{lema-1}
For any $\gamma, \xi \in \Gamma \cup \{0\}$ and $\alpha, \beta \in \Lambda \cup \{0\}$ the following assertions hold.
\begin{enumerate}
\item[{\rm i)}] $L_0 = H$.

\item[{\rm ii)}] If $[L_{\gamma}, L_{\xi}] \neq 0$ then $\gamma+\xi \in \Gamma \cup \{0\}$ and $[L_{\gamma}, L_{\xi}] \subset L_{\gamma + \xi}$.

\item[{\rm iii)}] If $A_{\alpha}A_{\beta} \neq 0$ then $\alpha+\beta \in \Lambda \cup \{0\}$ and $A_{\alpha}A_{\beta} \subset A_{\alpha + \beta}$.

\item[{\rm iv)}] If $A_{\alpha}L_{\gamma}\neq 0$ then $\alpha+\gamma \in \Gamma \cup \{0\}$ and $A_{\alpha}L_{\gamma} \subset L_{\alpha + \gamma}$.

\item[{\rm v)}] If $\rho(L_{\gamma})(A_{\alpha}) \neq 0$ then $\gamma+\alpha \in \Lambda \cup \{0\}$ and $\rho(L_{\gamma})(A_{\alpha}) \subset A_{\gamma + \alpha}.$
\end{enumerate}
\end{lemma}

\begin{proof}
i) and ii) are proved in \cite[Section 1]{YoLie}. %As direct consequence of the character of abelian subalgebra of $H$ we have $H \subset L_0.$ Let us prove $L_0 \subset H$. For any $0 \neq v \in L_0$ we can express $v = h \oplus (\bigoplus_{\gamma \in \Gamma} v_{\gamma})$ with $h \in H,$ any $v_{\gamma} \in L_{\gamma}$ for $\gamma \in \Gamma.$ Since $v \in L_0$ for any $h' \in H$ we have $[h',v] = 0$, then $$0 = [h',v] = \Bigl[h',\bigoplus_{\gamma \in \Gamma} v_{\gamma}\Bigr] = \bigoplus_{\gamma \in \Gamma} \gamma(h')v_{\gamma}.$$ From here, the fact that any $\gamma \neq 0$ gives that any $v_{\gamma} = 0$. Hence $v = h \in H$.

%ii) For each $h \in H, v_{\gamma} \in L_{\gamma}$ and $v_{\xi} \in L_{\xi}$, applying Jacobi identity we get
%\begin{eqnarray*}
%\begin{split}
%[h, [v_{\gamma},v_{\xi}]] &= -[v_{\gamma},[v_{\xi},h]] - [v_{\xi},[h,v_{\gamma}]] = [v_{\gamma},[h,v_{\xi}]] - [v_{\xi},[h,v_{\gamma}]]\\
%& = [v_{\gamma},\xi(h)v_{\xi}] -[v_{\xi},\gamma(h)v_{\gamma}] = (\gamma+\xi)(h)[v_{\gamma},v_{\xi}]
%\end{split}
%\end{eqnarray*}
%and we conclude $[v_{\gamma},v_{\xi}] \in L_{\gamma+\xi}$.
%\end{proof}

iii) Let $a_{\alpha} \in A_{\alpha}, a_{\beta} \in A_{\beta}.$ For any $h \in H$ we have that $\rho(h)$ is a derivation in $A.$ Then
\begin{eqnarray*}
\begin{split}
& \rho(h)(a_{\alpha}a_{\beta}) = \rho(h)(a_{\alpha}) a_{\beta} + a_{\alpha}\rho(h)(a_{\beta}) = \alpha(h)a_{\alpha} a_{\beta} + a_{\alpha}\beta(h)a_{\beta} = (\alpha + \beta)(h)(a_{\alpha}a_{\beta}).
\end{split}
\end{eqnarray*}
Therefore $a_{\alpha}a_{\beta} \in A_{\alpha+\beta}$.

iv) Let $a_{\alpha} \in A_{\alpha}, v_{\gamma} \in L_{\gamma}.$ By using Equation \eqref{fundamental} we get
\begin{eqnarray*}
\begin{split}
[h, a_{\alpha}v_{\gamma}] &= a_{\alpha}[h,v_{\gamma}] + \rho(h)(a_{\alpha})v_{\gamma} = a_{\alpha}(\gamma(h)v_{\gamma}) + \alpha(h)a_{\alpha}v_{\gamma}\\
&= (\gamma(h)+\alpha(h))a_{\alpha}v_{\gamma} = (\alpha + \gamma)(h)a_{\alpha}v_{\gamma}.
\end{split}
\end{eqnarray*}
So $a_{\alpha}v_{\gamma} \in L_{\alpha+\gamma}$.

\medskip

v) For $v_{\gamma} \in L_{\gamma}$ and $a_{\alpha} \in A_{\alpha}$ we have
\begin{eqnarray*}
\begin{split}
& \quad \rho(h)\bigl(\rho(v_{\gamma})(a_{\alpha})\bigr) = \bigl(\rho(h)\rho(v_{\gamma})\bigr)(a_{\alpha}) = \rho([h,v_{\gamma}])(a_{\alpha}) + \bigl(\rho(v_{\gamma})\rho(h)\bigr)(a_{\alpha})\\
& \quad \quad \quad = \rho\bigl(\gamma(h)v_{\gamma}\bigr)(a_{\alpha}) + \rho(v_{\gamma})\bigl(\rho(h)(a_{\alpha})\bigr) = \rho\bigl(\gamma(h)v_{\gamma}\bigr)(a_{\alpha}) + \rho(v_{\gamma})\bigl(\alpha(h)(a_{\alpha})\bigr)\\
& \quad \quad \quad = \gamma(h)\rho(v_{\gamma})(a_{\alpha}) + \alpha(h)\rho(v_{\gamma})(a_{\alpha}) = (\gamma + \alpha)(h)\rho(v_{\gamma})(a_{\alpha}),
\end{split}
\end{eqnarray*}
where the second equality comes from the fact that $\rho$ is a Lie algebra homomorphism. We proved $\rho(v_{\gamma})(a_{\alpha}) \in A_{\gamma+\alpha}$.
\end{proof}

\begin{remark}
Observe that Lemma \ref{lema-1}-iii) implies that  $A_0$ is a subalgebra of $A$.
\end{remark}

%%%%%%%%%%%%%%%%%%%%%%%%%%%%%%%%%%%%%%%%%%%%%%%%%5
%%%%%%%%%%%%%%%%%%%%%%%%%%%%%%%%%%%%%%%%%%%%%%%%%5
\section{Connections in the roots system of $L$. Decompositions of $L$}
%%%%%%%%%%%%%%%%%%%%%%%%%%%%%%%%%%%%%%%%%%%%%%%%%5
%%%%%%%%%%%%%%%%%%%%%%%%%%%%%%%%%%%%%%%%%%%%%%%%%5

Next we connect the set of nonzero roots of $L$ through nonzero roots of $L$ and nonzero weights of $A$,  considered both as elements in $H^*.$ We define $-\Gamma := \{-\gamma : \gamma \in \Gamma\}$ where $(-\gamma)(h):=-\gamma(h)$.  In a similar way we define  $-\Lambda := \{-\alpha : \alpha \in \Lambda\}.$ Finally, let us denote $$\hbox{ $\pm \Gamma := \Gamma \cup -\Gamma$ and $\pm \Lambda := \Lambda \cup -\Lambda.$}$$

\medskip

\noindent In the next definition the sum of elements in $\pm \Lambda \cup \pm \Gamma$ is taken in $H^*$.
\begin{definition}\label{connection}\rm
Let $\gamma, \xi \in \Gamma$. We say that $\gamma$ is {\it connected} to $\xi$ if either $\xi = \epsilon\gamma$ for some $\epsilon \in \{1,-1\}$, or there exists $\{\zeta_1,\zeta_2,\ldots,\zeta_n\}\subset \pm \Lambda \cup \pm \Gamma$, with $n \geq 2$, such that

\begin{enumerate}
\item[i)] $\zeta_1 = \gamma$.
\item[ii)] $\zeta_1+\zeta_2 \in \pm \Gamma,$

\vspace{0.05cm} $\vdots$

\hspace{-0.5cm} $\zeta_1 + \zeta_2 + \cdots + \zeta_{n-1} \in \pm \Gamma.$

\item[iii)] $\zeta_1 + \zeta_2 + \cdots + \zeta_n \in \{\xi, -\xi\}.$
\end{enumerate}

\noindent We also say that $\{\zeta_1,\ldots,\zeta_n\}$ is a {\it connection} from $\gamma$ to $\xi$.
\end{definition}

%  Observe that for any $\gamma \in \Gamma$, we have $\gamma$ is connected to $\gamma$ and also to $-\gamma$ in case $-\gamma \in \Gamma$.

%Our next goal is to show that the connection relation is of
%equivalence. We begin with a couple of previous lemmas.

\begin{proposition}\label{pro1}
The relation $\sim$ in $\Gamma$, defined by $\gamma \sim \xi$ if and only if $\gamma$ is connected to $\xi$, is an equivalence relation.
\end{proposition}

\begin{proof}
Similar to the proof of  \cite[Proposition 2.1]{YoLie}.
\end{proof}

\begin{remark}\label{Remark}
Let $\xi, \gamma \in \Gamma$ such that $\xi \sim \gamma.$ If $\gamma + \mu \in \Gamma,$ for $\mu \in \Lambda \cup \Gamma,$ then $\xi \sim \gamma + \mu$. Considering the connection $\{\gamma,\mu\}$ we get $\gamma \sim \gamma + \mu,$ and by transitivity $\xi \sim \gamma + \mu$. 

%Let $\gamma \in \Gamma$ and $\mu \in \Lambda \cup \Gamma$ such that $\gamma+\mu \in \Gamma,$ then considering the connection $\{\gamma,\mu\}$ we get $\gamma \sim \gamma + \mu$. Hence, if $\gamma \sim \xi$ for some $\xi \in \Gamma$, the transitivity of $\sim$ gives us $\xi \sim \gamma + \mu$.
\end{remark}

By Proposition \ref{pro1} the connection relation is an equivalence relation in $\Gamma$. From here, we can consider the quotient set $$\Gamma / \sim := \{[\gamma]: \gamma \in \Gamma\},$$ becoming $[\gamma]$ the set of nonzero roots of $L$ which are connected to $\gamma$. Our next goal is to associate an (adequate) ideal $I_{[\gamma]}$ of the Lie-Rinehart algebra $(L,A)$ to each $[\gamma]$. Fix $\gamma \in \Gamma$, we start by defining the set $L_{0,[\gamma]} \subset L_0$ as follows: $$L_{0,[\gamma]} := \bigl(\sum\limits_{\xi \in [\gamma], -\xi \in \Lambda}A_{-\xi}L_{\xi}\bigr) + \bigl(\sum\limits_{\xi \in [\gamma]}[L_{-\xi}, L_{\xi}]\bigr).$$ Next, we define $$L_{[\gamma]} := \bigoplus\limits_{\xi \in [\gamma]}L_{\xi}.$$ Finally, we denote by $I_{[\gamma]}$ the direct sum of the two subspaces above, that is, $$I_{[\gamma]} := L_{0,[\gamma]} \oplus L_{[\gamma]}.$$

\begin{proposition}\label{pro-2}
For any $[\gamma] \in \Gamma / \sim$, the following assertions hold.
\begin{enumerate}
\item[i)] $[I_{[\gamma]}, I_{[\gamma]}] \subset I_{[\gamma]}$.

\item[ii)] $AI_{[\gamma]} \subset I_{[\gamma]}.$

\item[iii)] $\rho(I_{[\gamma]})(A)L \subset I_{[\gamma]}.$
\end{enumerate}
\end{proposition}

\begin{proof}
i) Since $L_{0,{[\gamma]}} \subset L_0 =H$, then $[L_{0,{[\gamma]}}, L_{0,{[\gamma]}}] = 0$ and we have
\begin{equation}\label{cero1}
[I_{[\gamma]}, I_{[\gamma]}] = [L_{0,[\gamma]} \oplus L_{[\gamma]}, L_{0,[\gamma]} \oplus L_{[\gamma]}] \subset [L_{0,[\gamma]}, L_{[\gamma]}] + [L_{[\gamma]}, L_{[\gamma]}].
\end{equation}

Let us consider the first summand in Equation (\ref{cero1}). Given $\delta \in [\gamma]$ we have $[L_{0,[\gamma]}, L_{\delta}] \subset L_{\delta}$, hence $[L_{0,[\gamma]}, L_{\delta}] \subset L_{[\gamma]}$. Consider now the second summand. Given $\delta, \eta \in [\gamma]$ such that $[L_{\delta}, L_{\eta}] \neq 0$, then $[L_{\delta}, L_{\eta}] \subset L_{\delta + \eta}.$ If $\delta + \eta = 0$ we have $[L_{\delta}, L_{-\delta}] \subset L_{0,[\gamma]}.$
Suppose $\delta + \eta \in \Gamma,$ then by Remark \ref{Remark} we have $[L_{\delta}, L_{\eta}] \subset L_{\delta + \eta} \subset L_{[\gamma]}$. Hence $[I_{[\gamma]}, I_{[\gamma]}] \subset I_{[\gamma]}$.

ii) Observe that  $$AI_{[\gamma]} = \Bigl(A_0 \oplus \bigl(\bigoplus_{\alpha \in \Lambda}A_{\alpha}\bigr)\Bigr) \Bigl(\bigl(\sum\limits_{\xi \in [\gamma], -\xi \in \Lambda}A_{-\xi}L_{\xi}\bigr) + \bigl(\sum\limits_{\xi \in [\gamma]}[L_{-\xi}, L_{\xi}]\bigr) \oplus \bigoplus\limits_{\xi \in [\gamma]}L_{\xi}\Bigr).$$ We have to consider six cases:

$\bullet$ As $L$ is an $A$-module, for $\xi \in [\gamma]$ and $-\xi \in \Lambda$ we get $A_0(A_{-\xi}L_{\xi}) = (A_0A_{-\xi})L_{\xi} \subset A_{-\xi}L_{\xi} \subset L_{0,[\gamma]},$ using Lemma \ref{lema-1}-iii). That is,
\begin{equation}\label{1}
A_0(A_{-\xi}L_{\xi}) \subset L_{0,[\gamma]}.
\end{equation}

$\bullet$ For $\xi \in [\gamma]$, we have $A_0[L_{-\xi},L_{\xi}] \subset [L_{-\xi},A_0L_{\xi}] + \rho(L_{-\xi})(A_0)L_{\xi}$ by Equation \eqref{fundamental}. Since $A_0L_{\xi} \subset L_{\xi}$ we get $[L_{-\xi},A_0L_{\xi}] \subset [L_{-\xi},L_{\xi}]$. Also, by Lemma \ref{lema-1}-v) we obtain $\rho(L_{-\xi})(A_0) \subset A_{-\xi}$. If $A_{-\xi} \neq 0$ (otherwise is trivial), $-\xi \in \Lambda$ therefore $\rho(L_{-\xi})(A_0)L_{\xi} \subset A_{-\xi}L_{\xi}$ with $\xi \in [\gamma]$ and $-\xi \in \Lambda$. From here,
\begin{equation}\label{2}
A_0[L_{-\xi},L_{\xi}] \subset L_{0,[\gamma]}.
\end{equation}

$\bullet$ For $\xi \in [\gamma],$ from Lemma \ref{lema-1}-iv) it follows
\begin{equation}\label{3}
A_0L_{\xi} \subset L_{\xi} \subset L_{[\gamma]}.
\end{equation}

$\bullet$ For $\alpha \in \Lambda, \xi \in [\gamma]$ and $-\xi \in \Lambda,$ since $L$ is an $A$-module we get $A_{\alpha}(A_{-\xi}L_{\xi}) \subset (A_{\alpha}A_{-\xi})L_{\xi} \subset A_{\alpha-\xi}L_{\xi} \subset L_{\alpha}$, by Lemma \ref{lema-1} if $\alpha - \xi \in \Lambda$ (otherwise is trivial). If $L_{\alpha} \neq 0$ (otherwise is trivial) then $\alpha \in \Gamma$, and by Remark \ref{Remark} $\alpha \in [\gamma],$ that is,
\begin{equation}\label{4}
A_{\alpha}(A_{-\xi}L_{\xi}) \subset L_{[\gamma]}.
\end{equation}

$\bullet$ For $\alpha \in \Lambda, \xi \in [\gamma]$ we obtain $A_{\alpha}[L_{-\xi},L_{\xi}] \subset [L_{-\xi},A_{\alpha}L_{\xi}] + \rho(L_{-\xi})(A_{\alpha})L_{\xi}$. By Lemma \ref{lema-1}-iv) $A_{\alpha}L_{\xi} \subset L_{\alpha+\xi}.$ If $L_{\alpha+\xi}\neq 0$ (otherwise is trivial), $\alpha+\xi \in \Gamma,$ we get $[L_{-\xi},A_{\alpha}L_{\xi}] \subset [L_{-\xi},L_{\alpha+\xi}] \subset L_{\alpha}.$ By Remark \ref{Remark}, $\alpha + \xi \in [\gamma].$  If $\alpha \in \Gamma$ then $\alpha \sim \alpha+\xi,$ it follows $\alpha \in [\gamma]$. Also, by Lemma \ref{lema-1}-v) we have $\rho(L_{-\xi})(A_{\alpha}) \subset A_{-\xi + \alpha}$ and similarly to the  previous case $\rho(L_{-\xi})(A_{\alpha})L_{\xi} \subset L_{\alpha}$ with $\alpha \in [\gamma]$. We get
\begin{equation}\label{5}
A_{\alpha}[L_{-\xi},L_{\xi}] \subset L_{[\gamma]}.
\end{equation}

$\bullet$ For $\alpha \in \Lambda, \xi \in [\gamma]$ we obtain $A_{\alpha}L_{\xi} \subset L_{\xi+\alpha}$. Using again Remark \ref{Remark} we can prove $\xi+\alpha \in [\gamma],$ meaning that
\begin{equation}\label{6}
A_{\alpha}L_{\xi} \subset L_{[\gamma]}.
\end{equation}
From Equations \eqref{1}-\eqref{6}, assertion ii) is proved.

iii) By Equation \eqref{fundamental} and item ii) we get $$\rho(I_{[\gamma]})(A)L \subset [I_{[\gamma]},AL] + A[I_{[\gamma]},L] \subset I_{[\gamma]}.$$
\end{proof}

\begin{proposition}\label{pro-9}
Let  $[\gamma], [\delta] \in \Gamma / \sim$ with $[\gamma] \neq [\delta]$. Then $[I_{[\gamma]}, I_{[{\delta}]}] = 0$.
\end{proposition}

\begin{proof}
We have
\begin{equation}\label{cuatro1}
[I_{[\gamma]}, I_{[{\delta}]}] = [L_{0,[\gamma]} \oplus L_{[\gamma]}, L_{0,[\delta]} \oplus L_{[\delta]}] \subset [L_{0,[\gamma]} L_{[\delta]}] + [L_{[\gamma]}, L_{0,[\delta]}] +[L_{[\gamma]}, L_{[\delta]}].
\end{equation}

Consider the above third summand $[L_{[\gamma]}, L_{[\delta]}]$ and suppose there exist $\gamma_1 \in [\gamma]$ and $\delta_1 \in [\delta]$ such that $[L_{\gamma_1}, L_{\delta_1}]\neq 0$. As necessarily $\gamma_1 \neq -\delta_1$, then $\gamma_1 + \delta_1 \in \Gamma$. Since $\gamma \sim \gamma_1$ and $\gamma_1+\delta_1 \in \Gamma,$ by Remark \ref{Remark} we conclude $\gamma \sim \gamma_1 + \delta_1.$ Similarly we can prove $\delta \sim \gamma_1+\delta_1,$ so we conclude $\gamma \sim \delta,$ a contradiction. Hence $[L_{\gamma_1}, L_{\delta_1}] = 0$ and so
\begin{equation}\label{nueve1}
[L_{[\gamma]}, L_{[\delta]}] = 0.
\end{equation}

Consider now the first summand in Equation \eqref{cuatro1}, $$[L_{0,[\gamma]}, L_{[\delta]}] = \Bigl[\bigl(\sum_{\gamma_1 \in [\gamma], -\gamma_1 \in \Lambda}A_{-\gamma_1}L_{\gamma_1}\bigr) + \bigl(\sum_{\gamma_1 \in [\gamma]}[L_{-\gamma_1}, L_{\gamma_1}]\bigr),L_{[\delta]}\Bigr].$$

$\bullet$ For $\delta_1 \in [\delta]$  we obtain by Jacobi identity that $$[[L_{-\gamma_1},L_{\gamma_1}],L_{\delta_1}] = [[L_{\gamma_1},L_{\delta_1}],L_{-\gamma_1}] + [[L_{\delta_1},L_{-\gamma_1}],L_{\gamma_1}]$$ and by Equation \eqref{nueve1} that $$[L_{\gamma_1},L_{\delta_1}] = [L_{\delta_1},L_{-\gamma_1}] = 0.$$ Hence  $[[L_{-\gamma_1},L_{\gamma_1}],L_{\delta_1}] = 0$.

$\bullet$ If there exists $\delta_1 \in [\delta]$ such that $$0 \neq [A_{-\gamma_1}L_{\gamma_1},L_{\delta_1}] = [L_{\delta_1},A_{-\gamma_1}L_{\gamma_1}] \subset A_{-\gamma_1}[L_{\delta_1},L_{\gamma_1}] + \rho(L_{\delta_1})(A_{-\gamma_1})L_{\gamma_1},$$  we have $[L_{\delta_1},L_{\gamma_1}] = 0$ by Equation \eqref{nueve1}. Therefore $0 \neq \rho(L_{\delta_1})(A_{-\gamma_1})L_{\gamma_1} \subset A_{\delta_1-\gamma_1}L_{\gamma_1},$ and so $A_{\delta_1-\gamma_1}$ is nonzero. Since $\delta_1 - \gamma_1 \neq 0$ we have $\delta_1-\gamma_1 \in \Lambda,$ then the connection $\{\gamma_1,\delta_1-\gamma_1\}$ gives $\gamma_1 \sim \delta_1$, a contradiction. Consequently $$\rho(L_{\delta_1})(A_{-\gamma_1})L_{\gamma_1} = 0$$ and we have showed
\begin{equation}\label{nueve2}
[L_{0,[\gamma]},L_{\delta}] = 0.
\end{equation}
In a similar way, we get $[L_{[\gamma]},L_{0,[\delta]}] = 0.$ From Equations \eqref{cuatro1}-\eqref{nueve2}, we conclude $[I_{[\gamma]}, I_{[\delta]}] = 0.$
\end{proof}

\begin{theorem}\label{teo-1}
The following assertions hold.
\begin{enumerate}
\item[{\rm i)}] For any $[\gamma] \in \Gamma/ \sim$, the linear space $I_{[\gamma]} = L_{0,[\gamma]} \oplus L_{[\gamma]}$ associated to $[\gamma]$ is an ideal of $L$.

\item[{\rm ii)}] If $L$ is simple then all the roots of $\Gamma$ are connected. Moreover, $$H = \bigl(\sum\limits_{\gamma \in \Gamma, -\gamma \in \Lambda}A_{-\gamma}L_{\gamma}\bigr) + \bigl(\sum\limits_{\gamma \in \Gamma}[L_{-\gamma}, L_{\gamma}]\bigr).$$
\end{enumerate}
\end{theorem}

\begin{proof}
i) Since $H$ is abelian, $[I_{[\gamma]},H] \subset L_{[\gamma]} \subset I_{[\gamma]}$ and by Proposition \ref{pro-2}-i) and Proposition \ref{pro-9} we have $$[I_{[\gamma]}, L] = \bigl[I_{[\gamma]}, H \oplus (\bigoplus\limits_{\xi \in [\gamma]}L_{\xi}) \oplus (\bigoplus\limits_{\delta \notin [\gamma]}L_{\delta})\bigr] \subset I_{[\gamma]},$$ so $I_{[\gamma]}$ is a Lie ideal of $L.$ Clearly by Proposition \ref{pro-2}-ii) we also have that  $I_{[\gamma]}$ is an $A$-module. Finally, by Proposition \ref{pro-2}-iii) we conclude $I_{[\gamma]}$ is an ideal of $L.$

\medskip

ii) The simplicity of $L$ implies $I_{[\gamma]} \in \{{\rm Ker}\rho,L\}$ for any $\gamma \in \Gamma.$ If some $\gamma \in \Gamma$ is such that $I_{[\gamma]} = L,$ then $[\gamma] = \Gamma.$ Otherwise, if $I_{[\gamma]} = {\rm Ker}\rho$ for all $\gamma \in \Gamma$ then $[\gamma] = [\xi]$ for any $\gamma, \xi \in \Gamma$ and again $[\gamma] = \Gamma.$ Therefore in any case $L$ has all its nonzero roots connected and $H = \bigl(\sum_{\gamma \in \Gamma, -\gamma \in \Lambda}A_{-\gamma}L_{\gamma}\bigr) + \bigl(\sum_{\gamma \in \Gamma}[L_{-\gamma}, L_{\gamma}]\bigr).$
\end{proof}

\begin{theorem}\label{teo-2}
Let $(L,A)$ be a split Lie-Rinehart algebra. Then $$L = U + \sum\limits_{[\gamma] \in \Gamma/\sim}I_{[\gamma]},$$ where $U$ is a linear complement in $H$ of $\bigl(\sum_{\gamma \in \Gamma, -\gamma \in \Lambda}A_{-\gamma}L_{\gamma}\bigr) + \bigl(\sum_{\gamma \in \Gamma}[L_{-\gamma}, L_{\gamma}]\bigr)$ and any $I_{[\gamma]}$ is one of the ideals of $L$ described in Theorem \ref{teo-1}-i). Furthermore, $[I_{[\gamma]}, I_{[\delta]}] = 0$ when $[\gamma] \neq [\delta].$
\end{theorem}

\begin{proof}
We have $I_{[\gamma]}$ is well defined and, by Theorem \ref{teo-1}-i), an ideal of $L$, being clear that $$L = H \oplus (\bigoplus\limits_{\gamma \in \Gamma}L_{\gamma}) = U + \sum\limits_{[\gamma] \in \Gamma/\sim}I_{[\gamma]}.$$ Finally, Proposition \ref{pro-9} gives $[I_{[\gamma]}, I_{[\delta]}]=0$ if $[\gamma] \neq [\delta].$
\end{proof}

\noindent For a Lie-Rinehart algebra $L,$ we denote by ${\mathcal Z}(L) := \{v \in L : [v, L] = 0 \hbox{  and  } \rho(v) = 0\}$ the {\it center} of $L$.

\begin{corollary}\label{coro-1}
If ${\mathcal Z}(L) = 0$ and $H = \bigl(\sum_{\gamma \in \Gamma, -\gamma \in \Lambda}A_{-\gamma}L_{\gamma}\bigr) + \bigl(\sum_{\gamma \in \Gamma}[L_{-\gamma}, L_{\gamma}]\bigr)$ then $L$ is the direct sum of the ideals given in Theorem \ref{teo-1},
$$L = \bigoplus\limits_{[\gamma] \in \Gamma/\sim}I_{[\gamma]}.$$ Moreover, $[I_{[\gamma]}, I_{[\delta]}] = 0$ when $[\gamma] \neq [\delta].$
\end{corollary}

\begin{proof}
Since $H = \bigl(\sum\limits_{\gamma \in \Gamma, -\gamma \in \Lambda}A_{-\gamma}L_{\gamma}\bigr) + \bigl(\sum\limits_{\gamma \in \Gamma}[L_{-\gamma}, L_{\gamma}]\bigr)$ we get $$L = \sum\limits_{[\gamma] \in \Gamma/\sim} I_{[\gamma]}.$$
To verify the direct character of the sum, take some $v \in I_{[\gamma]} \cap \bigl(\sum_{[\delta]\in \Gamma/ \sim, [\delta] \neq [\gamma]}I_{[\delta]}\bigr)$. Since $v \in I_{[\gamma]},$ the fact $\bigl[I_{[\gamma]},I_{[\delta]}\bigr] = 0$ when $[\gamma] \neq [\delta]$ gives us $$[v,\sum_{[\delta] \in \Gamma/ \sim, [\delta] \neq [\gamma]}I_{[\delta]}] = 0.$$

\noindent In a similar way, since $v \in \sum_{[\delta]\in \Gamma/ \sim, [\delta] \neq [\gamma]}I_{[\delta]}$ we get $[v,I_{[\gamma]}] = 0.$ Therefore $[v,L]=0.$ Now,  Equation \eqref{fundamental} allows us to  conclude $\rho(v) = 0$. That is, $v \in {\mathcal Z}(L)$ and so $v=0$.
\end{proof}

%%%%%%%%%%%%%%%%%%%%%%%%%%%%%%%%%%%%%%%%%%%%%%%%%5
%%%%%%%%%%%%%%%%%%%%%%%%%%%%%%%%%%%%%%%%%%%%%%%%%5
\section{Connections in the weights system of $A$. Decompositions of $A$}
%%%%%%%%%%%%%%%%%%%%%%%%%%%%%%%%%%%%%%%%%%%%%%%%%5
%%%%%%%%%%%%%%%%%%%%%%%%%%%%%%%%%%%%%%%%%%%%%%%%%5

We begin this section by introducing an adequate notion of  connection among   the weights of $A$.

\begin{definition}\label{connection2}\rm
Let $\alpha, \beta \in \Lambda$. We say that $\alpha$ is {\it connected} to $\beta$ if either $\beta = \epsilon\alpha$ for some $\epsilon \in \{1,-1\}$, or there exists $\{\sigma_1,\sigma_2,\ldots,\sigma_n\} \subset \pm \Lambda \cup \pm \Gamma$, with $n \geq 2$, such that

\begin{enumerate}
\item[i)] $\sigma_1 = \alpha$.
\item[ii)] $\sigma_1 + \sigma_2 \in \pm\Lambda \cup \pm\Gamma,$

\vspace{0.05cm} $\vdots$

\hspace{-0.5cm} $\sigma_1 + \sigma_2 + \cdots + \sigma_{n-1} \in \pm\Lambda \cup \pm\Gamma.$

\item[iii)] $\sigma_1 + \sigma_2 + \cdots + \sigma_n \in \{\beta, -\beta\}.$
\end{enumerate}

\noindent We also say that $\{\sigma_1,\ldots,\sigma_n\}$ is a {\it connection} from $\alpha$ to $\beta$.
\end{definition}

%  Observe that for any $\alpha \in \Lambda$, we have $\alpha$ is connected to $\alpha$ and also to $-\alpha$ in case $-\alpha \in \Lambda$.
\noindent As in the previous section we can prove the next results.

\begin{proposition}\label{pro-1}
The relation $\approx$ in $\Lambda$, defined by $\alpha \approx\beta$ if and only if $\alpha$ is connected to $\beta$, is an equivalence relation.
\end{proposition}

\begin{remark}\label{Remark2}
Let $\alpha, \beta \in \Lambda$ such that $\alpha \approx \beta.$ If $\beta + \mu \in \Lambda,$ for $\mu \in \Lambda \cup \Gamma,$ then $\alpha \approx \beta + \mu$. Considering the connection $\{\beta,\mu\}$ we get $\beta \approx \beta + \mu,$ and by transitivity $\alpha \approx \beta + \mu$.
\end{remark}

\noindent By Proposition \ref{pro-1} the connection relation is an equivalence relation in $\Lambda$. From here, we can consider the quotient set $$\Lambda / \approx := \{[\alpha]: \alpha \in \Lambda \},$$ becoming $[\alpha]$ the set of nonzero weights which are connected to $\alpha$. Our next goal in this section is to associate an (adequate) ideal $\mathscr{A}_{[\alpha]}$ of the algebra $A$ to any $[\alpha] \in \Lambda / \approx$. Fix $\alpha \in \Lambda$, we start by defining the sets $$A_{0,[\alpha]} := \bigl(\sum_{-\beta \in \Gamma, \beta \in [\alpha]}\rho(L_{-\beta})(A_{\beta})\bigr) + \bigl(\sum_{\beta \in [\alpha]}A_{-\beta}A_{\beta}\bigr) \subset A_0$$ and $$A_{[\alpha]} := \bigoplus \limits_{\beta \in [\alpha]} A_{\beta}.$$
  Hence, we denote by $\mathscr{A}_{[\alpha]}$ the direct sum of the two subspaces above. That is, $$\mathscr{A}_{[\alpha]} := A_{0,[\alpha]} \oplus A_{[\alpha]}.$$

\begin{proposition}\label{pro2}
For any $[\alpha] \in \Lambda / \approx$ we have $\mathscr{A}_{[\alpha]}\mathscr{A}_{[\alpha]} \subset \mathscr{A}_{[\alpha]}$.
\end{proposition}

\begin{proof}
Since the algebra $A$ is commutative we have
\begin{equation}\label{cero}
\mathscr{A}_{[\alpha]}\mathscr{A}_{[\alpha]} = \Bigl(A_{0,[\alpha]} \oplus A_{[\alpha]}\Bigr)\Bigl(A_{0,[\alpha]} \oplus A_{[\alpha]}\Bigl) \subset A_{0,[\alpha]}A_{0,[\alpha]} + A_{0,[\alpha]}A_{[\alpha]} + A_{[\alpha]}A_{[\alpha]}.
\end{equation}

Let us consider the second summand in Equation (\ref{cero}). Given $\beta \in [\alpha]$ we have $A_{0,[\alpha]}A_{\beta} \subset A_0A_{\beta} \subset A_{\beta}$, by Lemma \ref{lema-1}-iii). Hence
\begin{equation}\label{ceroo}
A_{0,[\alpha]}A_{\beta} \subset A_{[\alpha]}.
\end{equation}

For the third summand in Equation (\ref{cero}), given $\beta, \nu \in [\alpha]$ such that $0 \neq A_{\beta}A_{\nu} \subset A_{\beta+\nu}.$ If $\beta + \nu = 0$ we have $A_{-\beta}A_{\beta} \subset A_0$ and so $A_{-\beta}A_{\beta} \subset A_{0,[\alpha]}.$ Suppose $\beta + \nu \in \Lambda$, then by Remark \ref{Remark2} we have $\beta + \nu \in [\alpha]$ and so $A_{\beta}A_{\nu} \subset A_{\beta+\nu} \subset A_{[\alpha]}$. Hence $(\bigoplus_{\beta \in [\alpha]}A_{\beta})(\bigoplus_{\nu \in [\alpha]}A_{\nu}) \subset A_{0,[\alpha]} \oplus A_{[\alpha]}$. That is,
\begin{equation}\label{eq0.5}
A_{[\alpha]}A_{[\alpha]} \subset \mathscr{A}_{[\alpha]}.
\end{equation}

Finally we consider the first summand $A_{0,[\alpha]}A_{0,[\alpha]}$ and suppose there exist $\beta, \nu \in [\alpha]$ such that $$\Bigl(\rho(L_{-\beta})(A_{\beta}) + A_{-\beta}A_{\beta}\Bigr)\Bigl(\rho(L_{-\nu})(A_{\nu}) + A_{-\nu}A_{\nu}\Bigr) \neq 0,$$ so
\begin{eqnarray}
&& \rho(L_{-\beta})(A_{\beta})\rho(L_{-\nu})(A_{\nu}) + \rho(L_{-\beta})(A_{\beta})(A_{-\nu}A_{\nu}) \nonumber \\
&& \quad \quad + (A_{-\beta}A_{\beta})\rho(L_{-\nu})(A_{\nu})
+ (A_{-\beta}A_{\beta})(A_{-\nu}A_{\nu}) \neq 0 \label{ideal_A}
\end{eqnarray}
For the last summand in Equation \eqref{ideal_A}, in case $\nu \neq -\beta,$ by the commutativity and associativity of $A$ we have $$(A_{-\beta}A_{\beta})(A_{-\nu}A_{\nu}) = (A_{-\beta}A_{-\nu})(A_{\beta}A_{\nu}) \subset A_{-(\beta+\nu)}A_{(\beta+\nu)} \subset A_{0,[\alpha]}$$ by Remark \ref{Remark2}. In case $\nu = -\beta,$ it follows $$(A_{-\beta}A_{\beta})(A_{\beta}A_{-\beta}) = A_{-\beta}(A_{\beta}A_{\beta}A_{-\beta}) \subset A_{-\beta}A_{\beta} \subset A_{0,[\alpha]}.$$ For the second summand in Equation \eqref{ideal_A}, $\rho(L_{-\beta})(A_{\beta})(A_{-\nu}A_{\nu}),$ since $\rho(L_{-\beta})$ is a derivation in $A$ we get
\begin{eqnarray*}
\begin{split}
& \rho(L_{-\beta})(A_{\beta})(A_{-\nu}A_{\nu}) \subset \rho(L_{-\beta})(A_{\beta}(A_{-\nu}A_{\nu})) + A_{\beta}\rho(L_{-\beta})(A_{-\nu}A_{\nu})
\end{split}
\end{eqnarray*}
but $\rho(L_{-\beta})(A_{\beta}(A_{-\nu}A_{\nu})) \subset \rho(L_{-\beta})(A_{\beta})$ and $A_{\beta}\rho(L_{-\beta})(A_{-\nu}A_{\nu}) \subset A_{\beta}A_{-\beta}$ so $$\rho(L_{-\beta})(A_{\beta})(A_{-\nu}A_{\nu}) \subset \rho(L_{-\beta})(A_{\beta}) + A_{-\beta}A_{\beta} \subset A_{0,[\alpha]}.$$ By commutativity we also get the summand $(A_{-\beta}A_{\beta})\rho(L_{-\nu})(A_{\nu}) \subset A_{0,[\alpha]}$. Finally, for the first summand, since $\rho(L_{-\beta})$ is a derivation, we have $$\rho(L_{-\beta})\bigl(A_{\beta}\rho(L_{-\nu})(A_{\nu})\bigr) \subset \rho(L_{-\beta})(A_{\beta})\rho(L_{-\nu})(A_{\nu}) + A_{\beta}\rho(L_{-\beta})\bigl(\rho(L_{-\nu})(A_{\nu})\bigr).$$ As $\rho(L_{-\beta}) \bigl(A_{\beta}\rho(L_{-\nu})(A_{\nu})\bigr) \subset \rho(L_{-\beta})(A_{\beta})$ and, by associativity, $$A_{\beta}\rho(L_{-\beta})\bigl(\rho(L_{-\nu})(A_{\nu})\bigr) \subset A_{\beta}A_{-\beta}$$ then $$\rho(L_{-\beta})(A_{\beta})\rho(L_{-\nu})(A_{\nu}) \subset \rho(L_{-\beta})(A_{\beta}) + A_{\beta}A_{-\beta} \subset A_{0,[\alpha]}.$$

We have showed
\begin{equation}\label{eq0.6}
A_{0,[\alpha]}A_{0,[\alpha]} \subset A_{0,[\alpha]} \subset \mathscr{A}_{[\alpha]}.
\end{equation}

From Equations \eqref{cero}-\eqref{eq0.5} and \eqref{eq0.6} we get $\mathscr{A}_{[\alpha]}\mathscr{A}_{[\alpha]} \subset \mathscr{A}_{[\alpha]}.$
\end{proof}

\begin{proposition}\label{pro9}
For any $[\alpha] \neq [\psi]$ we have $\mathscr{A}_{[\alpha]}\mathscr{A}_{[\psi]}=0$.
\end{proposition}

\begin{proof}
We have
\begin{equation}\label{cuatro}
\Bigl(A_{0,[\alpha]} \oplus A_{[\alpha]}\Bigr)\Bigl(A_{0,[\psi]} \oplus A_{[\psi]}]\Bigr) \subset A_{0,[\alpha]}A_{0,[\psi]} + A_{0,[\alpha]}A_{[\psi]} + A_{[\alpha]}A_{0,[\psi]} + A_{[\alpha]}A_{[\psi]}.
\end{equation}

Consider the above fourth summand $A_{[\alpha]}A_{[\psi]}$ and suppose there exist $\alpha_1 \in [\alpha]$ and $\psi_1 \in [\psi]$ such that $A_{\alpha_1}A_{\psi_1} \neq 0$, so $A_{\alpha_1 +\psi_1} \neq 0$. Then $\alpha_1 + \psi_1  \in \Lambda \cup \{0\}$. As necessarily $\alpha_1 \neq -\psi_1$, it follows that $\alpha_1 + \psi_1  \in \Lambda$. By Remark \ref{Remark2}, $\alpha \sim \alpha_1 + \psi_1$ and $\psi \sim \alpha_1 + \psi_1,$ and by equivalence relation we have $[\alpha] = [\psi]$, a contradiction. Hence $A_{\alpha_1}A_{\psi_1} = 0$ and so
\begin{equation}\label{nueve}
A_{[\alpha]}A_{[\psi]} = 0.
\end{equation}

Consider now the second summand $A_{0,[\alpha]}A_{[\psi]}$ in Equation (\ref{cuatro}). We take $\alpha_1 \in [\alpha]$ and $\psi_1 \in [\psi]$ such that $$\Bigl(\rho(L_{-\alpha_1})(A_{\alpha_1})A_{\psi_1} + A_{-\alpha_1}A_{\alpha_1}\Bigr)A_{\psi_1} \neq 0.$$
Suppose $(A_{-\alpha_1}A_{\alpha_1})A_{\psi_1} \neq 0.$ By using associativity of $A$ we get $A_{-\alpha_1}(A_{\alpha_1}A_{\psi_1}) \neq 0,$ so $A_{\alpha_1+\psi_1} \neq 0$ and then $\alpha_1+\psi_1 \in \Lambda \cup \{0\}$. Arguing as above $\alpha \approx \psi$, a contradiction. If the another summand $\rho(L_{-\alpha_1})(A_{\alpha_1})A_{\psi_1} \neq 0$, since $\rho(L_{-\alpha_1})$ is a derivation then $\rho(L_{-\alpha_1})(A_{\alpha_1}A_{\psi_1})$ or $A_{\alpha_1}\rho(L_{-\alpha_1})(A_{\psi_1})$ is nonzero, but in any case we argue similarly as above to get $\alpha \approx \psi,$ a contradiction. From here
\begin{equation}\label{lex}
A_{0,[\alpha]}A_{[\psi]} = 0.
\end{equation}
By commutativity, $A_{[\alpha]}A_{0,[\psi]} = 0.$

Finally, let us prove $A_{0,[\alpha]}A_{0,[\psi]} = 0$. Suppose there exist $\alpha_1 \in [\alpha], \psi_1 \in [\psi]$ such that $$\rho(L_{-\alpha_1})(A_{\alpha_1})\rho(L_{-\psi_1})(A_{\psi_1}) + \rho(L_{-\alpha_1})(A_{\alpha_1})(A_{-\psi_1}A_{\psi_1})$$ $$+ (A_{-\alpha_1}A_{\alpha_1})\rho(L_{-\psi_1})(A_{\psi_1}) + (A_{-\alpha_1}A_{\alpha_1})(A_{-\psi_1}A_{\psi_1}) \neq 0.$$ We can argue as in the proof of Proposition \ref{pro2} to obtain
\begin{equation}\label{tex}
A_{0,[\alpha]}A_{0,[\psi]} = 0.
\end{equation}
From Equations \eqref{cuatro}-\eqref{tex} we conclude $\mathscr{A}_{[\alpha]}\mathscr{A}_{[\psi]} = 0$.
\end{proof}

\noindent We recall that a subspace $I$ of a commutative algebra $A$ is called an {\it ideal} of $A$ if $AI \subset I$. We say that $A$ is {\it simple} if $AA \neq 0$ and it contains no proper ideals.

\begin{theorem}\label{teo-11}
Let $A$ be a commutative and associative algebra associated to a Lie-Rinehart algebra $L.$ Then the following assertions hold.
\begin{enumerate}
\item[{\rm i)}] For any $[\alpha] \in \Lambda/ \approx$, the linear space $$\mathscr{A}_{[\alpha]} = A_{0,[\alpha]} \oplus A_{[\alpha]}$$ of $A$ associated to $[\alpha]$ is an ideal of $A$.

\item[{\rm ii)}] If $A$ is simple then all weights of $\Lambda$ are connected. Furthermore, $$A_0 = \bigl(\sum\limits_{-\alpha \in \Gamma, \alpha \in \Lambda}\rho(L_{-\alpha})(A_{\alpha})\bigr) + \bigl(\sum\limits_{\alpha \in \Lambda}A_{-\alpha}A_{\alpha}\bigr).$$
\end{enumerate}
\end{theorem}

\begin{proof}
i) Since $\mathscr{A}_{[\alpha]}A_0 \subset \mathscr{A}_{[\alpha]}$ (by associativity of $A$), Propositions \ref{pro2} and \ref{pro9} allow us to assert  $$\mathscr{A}_{[\alpha]}A = \mathscr{A}_{[\alpha]}\Bigl(A_0 \oplus (\bigoplus\limits_{\beta \in [\alpha]}A_{\beta}) \oplus (\bigoplus\limits_{\psi \notin [\alpha]}A_{\psi})\Bigr) \subset \mathscr{A}_{[\alpha]}.$$ We conclude $\mathscr{A}_{[\alpha]}$ is an ideal of $A$.

\medskip

ii) The simplicity of $A$ implies $\mathscr{A}_{[\alpha]} = A$, for any $\alpha \in \Lambda$. From here, it is clear that $[\alpha] = \Lambda$ and $A_0 = \sum\limits_{-\alpha \in \Gamma, \alpha \in \Lambda}\rho(L_{-\alpha})(A_{\alpha}) + \sum\limits_{\alpha \in \Lambda}A_{-\alpha}A_{\alpha}$.
\end{proof}

\begin{theorem}\label{teo2}
Let $A$ be a commutative and associative algebra associated to a Lie-Rinehart algebra $L.$  Then $$A = V + \sum\limits_{[\alpha] \in \Lambda/\approx}\mathscr{A}_{[\alpha]},$$ where $V$ is a linear complement in $A_0$ of $\bigl(\sum_{-\alpha \in \Gamma, \alpha \in \Lambda}\rho(L_{-\alpha})(A_{\alpha})\bigr) + \bigl(\sum_{\alpha \in \Lambda}A_{-\alpha}A_{\alpha}\bigr)$ and any $\mathscr{A}_{[\alpha]}$ is one of the ideals of $A$ described in Theorem \ref{teo-11}-i). Furthermore, $\mathscr{A}_{[\alpha]}\mathscr{A}_{[\psi]} = 0$ when $[\alpha] \neq [\psi].$
\end{theorem}

\begin{proof}
We know that $\mathscr{A}_{[\alpha]}$ is well defined and, by Theorem \ref{teo-11}-i), an ideal of $A$, being clear that $$A = A_0 \oplus (\bigoplus\limits_{\alpha \in \Lambda}A_{\alpha}) = V + \sum\limits_{[\alpha] \in \Lambda/\approx}\mathscr{A}_{[\alpha]}.$$ Finally, Proposition \ref{pro9} gives $\mathscr{A}_{[\alpha]}\mathscr{A}_{[\psi]}=0$ if $[\alpha] \neq [\psi].$
\end{proof}

\medskip

\noindent Let us denote by ${\mathcal Z}(A) := \{a \in A : aA = 0\}$ the {\it center} of the algebra $A$.
% We recall that $A$ is called {\it perfect} if $\mathcal{Z}(A) = 0$ and {\bf ES PRECISA ESTA SEGUNDA CONDICION? $AA = A$}.

\begin{corollary}\label{coro1}
Let
%$A$ be the commutative and associative algebra associated to
 $(L,A)$  be a Lie-Rinehart algebra. If  ${\mathcal Z}(A) = 0$ and $$A_0 = \bigl(\sum\limits_{-\alpha \in \Gamma, \alpha \in \Lambda}\rho(L_{-\alpha})(A_{\alpha})\bigr) + \bigl(\sum\limits_{\alpha \in \Lambda}A_{-\alpha}A_{\alpha}\bigr),$$ then $A$ is the direct sum of the  ideals given in Theorem \ref{teo-11}-i),
$$A = \bigoplus\limits_{[\alpha] \in \Lambda/\approx}\mathscr{A}_{[\alpha]}.$$ Furthermore, $\mathscr{A}_{[\alpha]}\mathscr{A}_{[\psi]} = 0$ when $[\alpha] \neq [\psi].$
\end{corollary}

\begin{proof}
Since $A_0 = \bigl(\sum_{-\alpha \in \Gamma, \alpha \in \Lambda}\rho(L_{-\alpha})(A_{\alpha})\bigr) + \bigl(\sum_{\alpha \in \Lambda}A_{-\alpha}A_{\alpha}\bigr)$ we obtain $A = \sum_{[\alpha] \in \Lambda/\approx} \mathscr{A}_{[\alpha]}$.
To verify the direct character of the sum, take some $$a \in \mathscr{A}_{[\alpha]} \cap (\sum\limits_{[\psi]\in \Lambda/ \approx,
[\psi] \neq [\alpha]}\mathscr{A}_{[\psi]}).$$  Since $a \in \mathscr{A}_{[\alpha]}$, the fact $\mathscr{A}_{[\alpha]}\mathscr{A}_{[\psi]}=0$ when $[\alpha] \neq [\psi]$ gives us
$$a(\sum_{[\psi]\in \Lambda/ \approx, [\psi] \neq
[\alpha]}\mathscr{A}_{[\psi]}) = 0.$$ In a similar way, since $a \in \sum_{[\psi]\in \Lambda/ \approx, [\psi] \neq [\alpha]}\mathscr{A}_{[\psi]}$ we get $a\mathscr{A}_{[\alpha]}=0.$ That is, $a \in {\mathcal Z}(A)$ and so $a=0$.
\end{proof}

%%%%%%%%%%%%%%%%%%%%%%%%%%%%%%%%%%%%%%%%%%%%%%%%%%%%%%%%%%
%%%%%%%%%%%%%%%%%%%%%%%%%%%%%%%%%%%%%%%%%%%%%%%%%%%%%%%%%%
\section{Relating the decompositions of $L$ and $A$}
%%%%%%%%%%%%%%%%%%%%%%%%%%%%%%%%%%%%%%%%%%%%%%%%%%%%%%%%%%
%%%%%%%%%%%%%%%%%%%%%%%%%%%%%%%%%%%%%%%%%%%%%%%%%%%%%%%%%%

The aim of this section is to show that the decompositions of $L$ and $A$ as direct sum of ideals, given in Sections 2 and 3 respectively, are closely related.

\begin{definition}\label{tight}\rm
A split Lie-Rinehart algebra $(L,A)$ is {\it tight} if ${\mathcal Z}(L)={\mathcal Z}(A)=0,$ $AA = A,$ $AL = L$ and
\begin{eqnarray*}
&& H = \bigl(\sum\limits_{\gamma \in \Gamma, -\gamma \in \Lambda}A_{-\gamma}L_{\gamma}\bigr) + \bigl(\sum\limits_{\gamma \in \Gamma}[L_{-\gamma}, L_{\gamma}]\bigr),\\
&& A_0 = \bigl(\sum\limits_{-\alpha \in \Gamma, \alpha \in \Lambda}\rho(L_{-\alpha})(A_{\alpha})\bigr) + \bigl(\sum\limits_{\alpha \in \Lambda}A_{-\alpha}A_{\alpha}\bigr).
\end{eqnarray*}
\end{definition}

\noindent If $(L,A)$ is tight then Corollaries \ref{coro-1} and \ref{coro1} say that $$\hbox{$L = \bigoplus\limits_{[\gamma] \in \Gamma/\sim} I_{[\gamma]}$ \hspace{0.2cm} and \hspace{0.2cm}  $A =\bigoplus\limits_{[\alpha] \in \Lambda/\approx} \mathscr{A}_{[\alpha]}$},$$ with any $I_{[\gamma]}$ an ideal of $L$ verifying $[I_{[\gamma]},I_{[\delta]}]=0$ if $[\gamma] \neq [\delta]$ and any $\mathscr{A}_{[\alpha]}$ an ideal of $A$ satisfying $\mathscr{A}_{[\alpha]}\mathscr{A}_{[\psi]}=0$ if $[\alpha] \neq [\psi]$.

\begin{proposition}
Let $(L,A)$ be a tight split Lie-Rinehart algebra. Then for any $[\gamma] \in \Gamma/\sim$ there exists a unique $[\alpha] \in \Lambda/\approx$ such that $\mathscr{A}_{[\alpha]}I_{[\gamma]} \neq 0$.
\end{proposition}

\begin{proof}
First we prove the existence. Given $[\gamma] \in \Gamma/\sim,$ let us suppose that $AI_{[\gamma]} = 0$. Since $I_{[\gamma]}$ is an ideal it follows $$\hbox{$[I_{[\gamma]},AL] = [I_{[\gamma]}, \bigoplus\limits_{\xi \in \Gamma / \sim} AI_{[\xi]}] = [I_{[\gamma]},AI_{[\gamma]}] = 0.$}$$ By hypothesis $AL = L,$ then $I_{[\gamma]} \subset \mathcal{Z}(L) = \{0\},$ a contradiction. Since $A =\bigoplus_{[\alpha] \in \Lambda/\approx} \mathscr{A}_{[\alpha]},$ there  exists $[\alpha] \in \Lambda/\approx$ such that $\mathscr{A}_{[\alpha]}I_{[\gamma]} \neq 0$.

Now we prove that $[\alpha]$ is unique. Suppose that $\beta$ is another weight of $A$ which satisfies ${\mathcal A}_{[\beta]}I_{[\gamma]} \neq 0$. From $\mathscr{A}_{[\alpha]}I_{[\gamma]} \neq 0$ and ${\mathcal A}_{[\beta]}I_{[\gamma]} \neq 0$ we can take $\alpha' \in {[\alpha]}, \beta' \in {[\beta]}$ and $\gamma',\gamma'' \in {[\gamma]}$ such that $A_{\alpha'}L_{\gamma'} \neq 0$ and $A_{\beta'}L_{\gamma''} \neq 0$. Since $\gamma',\gamma'' \in {[\gamma]}$, we can fix a connection $$\{\gamma', \zeta_2,\ldots,\zeta_n\},$$ from $\gamma'$ to $\gamma''$.

We have to distinguish four cases. First, $\alpha'+\gamma' \neq 0$ and  $\beta'+\gamma'' \neq 0$. Then $\alpha'+\gamma'$,
$\beta'+\gamma'' \in \Gamma$, and so $\alpha'$ is connected to $\beta'.$ Indeed, in the case $\gamma'+ \zeta_2 + \cdots +\zeta_n = \gamma'',$ the connection from $\alpha'$ to $\beta'$ is $$\{\alpha',\gamma', -\alpha',\zeta_2,\ldots,\zeta_n,\beta',-\gamma''\} \subset \pm \Lambda \cup \pm \Gamma.$$ While in the case $\gamma'+ \zeta_2 + \cdots +\zeta_n= -\gamma''$ the connection is $$\{\alpha',\gamma', -\alpha',\zeta_2,\ldots,\zeta_n, -\beta',\gamma''\} \subset \pm \Lambda \cup \pm \Gamma.$$ From here $\alpha' \approx \beta'$ and so $[\alpha]=[\beta]$. In the second case,  $\alpha'+\gamma' = 0$ and  $\beta'+\gamma'' \neq 0$. Hence   $\alpha'=-\gamma'$,
$\beta'+\gamma'' \in \Gamma$ and then  $$\{-\gamma',-\zeta_2,\ldots,-\zeta_n,-\beta', \gamma''\} \subset \pm \Lambda \cup \pm \Gamma$$ is a connection from $\alpha'$ to $\beta'$ in the case $\gamma'+ \zeta_1+ \cdots +\zeta_n= \gamma''$ while $$\{-\gamma',-\zeta_2,\ldots,-\zeta_n,\beta',-\gamma''\} \subset \pm \Lambda \cup \pm \Gamma$$ is a connection in the case $\gamma'+ \zeta_1+ \cdots +\zeta_n= -\gamma''$. From here, $[\alpha]=[\beta]$. In the third case we suppose  $\alpha'+\gamma' \neq 0$ and
$\beta'+\gamma''=0$. We can argue as in the previous case to get
$[\alpha]=[\beta]$. Finally, in the fourth case we consider
$\alpha'+\gamma'=0, \beta'+\gamma'' = 0$. Hence
$\alpha'=-\gamma', \beta'=-\gamma''.$ Then
$$\{-\gamma',-\zeta_2,\ldots,-\zeta_n\}$$ is a connection
between $\alpha'$ and $\beta'$ which implies $[\alpha]=[\beta]$. We conclude $[\alpha] \in \Lambda/\approx$ is the unique element in $\Lambda/\approx$  such that $\mathscr{A}_{[\alpha]}I_{[\gamma]} \neq 0$ for the given $[\gamma] \in \Gamma/\sim$.
\end{proof}

  Observe that the  above proposition shows that $I_{[\gamma]}$ is an $\mathscr{A}_{[\alpha]}$-module. Hence we can assert the following result.

\begin{theorem}\label{lema_final}
Let $(L,A)$ be a tight split Lie-Rinehart algebra. Then $$\hbox{$L =\bigoplus\limits_{i\in I}L_i$ \hspace{0.4cm} and \hspace{0.4cm} $A = \bigoplus\limits_{j \in J}A_j$}$$ with any $L_i$ a nonzero ideal of $L$ and any $A_j$ a nonzero ideal of $A$,  in such a way that for any $i \in I$ there exists a unique $\tilde{i} \in J$ such that $$A_{\tilde{i}}L_i \neq 0.$$
% Furthermore, any $L_i$ is a split Lie-Rinehart algebra over $A_{\tilde{i}}$ and $(L,A)$ is the external direct sum of the family of split Lie-Rinehart algebras $\{(L_i,A_{\tilde{i}})\}_{i \in I}$.
\end{theorem}

\medskip

%%%%%%%%%%%%%%%%%%%%%%%%%%%%%%%%%%%%%%%%%%%%%%%%%%%%%%%%%%%%%%
\section{Decompositions through the families of the  simple ideals}
%%%%%%%%%%%%%%%%%%%%%%%%%%%%%%%%%%%%%%%%%%%%%%%%%%%%%%%%%%%%%%

In this section we are going to show that, under mild conditions, the decomposition of a  split Lie-Rinehart algebra $(L,A)$ given in Theorem  \ref{lema_final} can be obtained by means of the families of the  simple ideals of $L$ and $A$. In this section we always suppose that $\Gamma$ and $\Lambda$ are {\it symmetric}, that is, $\Gamma = -\Gamma$ and $\Lambda = -\Lambda$, respectively.

%We begin with a result that holds for any split Lie algebra.

%\begin{lemma}\label{lema-4}
%Let $L$ be a split Lie algebra such that ${\mathcal Z}(L)=\{0\}$. If $I$ is an ideal of $L$ contained in $H,$ then $I = \{0\}$.
%\end{lemma}
%\begin{proof}
%Consequence of $[I,H] \subset [H,H] = 0$ and $[I, \bigoplus \limits_{\gamma \in {\Gamma}} L_{\gamma}] \subset (\bigoplus \limits_{\gamma \in {\Gamma}} L_{\gamma}) \cap H = 0$.
%\end{proof}

\noindent Let us introduce the concepts of root-multiplicativity and maximal length in the framework of split Lie-Rinehart algebras, in a similar way to the ones for other classes of split algebras, such as split Lie algebras, split Malcev algebras, split Leibniz algebras and split Hom-algebras (see \cite{YoLie,Yopoisson,AMSi,Cao8} for these notions and examples).

\begin{definition}\rm
We say that a split Lie-Rinehart algebra $(L,A)$ is {\it
root-multiplicative} if for any  $\gamma, \delta \in \Gamma$ and $\alpha, \beta \in \Lambda$  the following conditions hold.
\begin{itemize}
\item If $\gamma+\delta \in \Gamma$ then $[L_{\gamma}, L_{\delta}] \neq 0.$

\item If $\alpha + \gamma \in \Gamma$ then $A_{\alpha}L_{\gamma} \neq 0.$

\item If $\alpha + \beta  \in \Lambda$ then $A_{\alpha}A_{\beta} \neq 0.$
\end{itemize}
\end{definition}

\begin{definition}\rm
A split Lie-Rinehart algebra $(L,A)$ is called of {\it maximal length} if $\dim L_{\gamma} = \dim A_{\alpha}=1$ for any $\gamma \in \Gamma$ and $\alpha \in \Lambda$.
\end{definition}

\noindent Observe that if $L$ and $A$ are  simple algebras  then $\mathcal{Z}(L) = \mathcal{Z}(A) =\{0\}$. Also as consequence of  Theorem \ref{teo-1}-ii)  and Theorem   \ref{teo-11}-ii) we get that all of the nonzero roots in $\Gamma$ are connected, that all of the nonzero weights in $\Lambda$ are also connected and that $H = \bigl(\sum_{\gamma \in \Lambda \cap \Gamma}A_{-\gamma}L_{\gamma}\bigr) + \bigl(\sum_{\gamma \in \Gamma}[L_{\gamma},L_{-\gamma}]\bigr)$ and $A_0 = \bigl(\sum_{-\alpha \in \Gamma, \alpha \in \Lambda}\rho(L_{-\alpha})(A_{\alpha})\bigr) + \bigl(\sum_{\alpha \in \Lambda}A_{-\alpha}A_{\alpha}\bigr).$ From here, the conditions for $(L,A)$ of being tight (see Definition \ref{tight}) together with the ones of having $\Gamma$ and $\Lambda$  all of their elements connected,  are  necessary conditions to get a characterization of the simplicity of the algebras $L$ and $A$. Actually, under the hypothesis of being $(L,A)$ of maximal length and root-multiplicative, these are also sufficient conditions as Theorem \ref{teo100} shows.

%In next result we approach to a partial converse of this result under additional conditions.

\begin{proposition}\label{teo-3}
Let $(L,A)$ be a tight split Lie-Rinehart algebra of maximal length, root-multiplicative and all its nonzero roots are connected. Then either $L$ is simple or $L = I \oplus I'$ where $I$ and $I'$ are simple ideals of $L$.
\end{proposition}

\begin{proof}
Consider $I$ a nonzero ideal of $L$. In case $I \subset H$, on the one hand $[I,H] \subset [H,H] = 0,$ and on the other hand $[I, \bigoplus_{\gamma \in {\Gamma}} L_{\gamma}] \subset (\bigoplus_{\gamma \in {\Gamma}} L_{\gamma}) \cap H = 0.$ So $I \subset \mathcal{Z}(L) = \{0\},$ a contradiction. Then $I \not\subset H$ and by \cite[Lemma 3.2]{YoLie} we can write $$I = (I \cap H) \oplus (\bigoplus_{\gamma \in \Gamma}(I \cap L_{\gamma}))$$ with $(I \cap L_{\gamma}) \neq 0$ for at least one $\gamma \in \Gamma$. \black Let us  denote by $I_{\gamma} := I \cap L_{\gamma}$ and by $ \Gamma_I := \{\gamma \in \Gamma: I_{\gamma} \neq 0\}$. Then we can write $I = (I \cap H) \oplus (\bigoplus_{\gamma \in \Gamma_I}I_{\gamma}).$ Let us distinguish two cases.

In the first case assume there exists $\gamma \in \Gamma_I$ such that $-\gamma \in \Gamma_I$. Then $0 \neq I_{\gamma} \subset I$ and we can assert by the maximal length of $(L,A)$ that
\begin{equation}\label{I-1}
L_{\gamma} \subset I.
\end{equation}
Now,  take some $\delta \in \Gamma$ satisfying $\delta \notin \{\gamma,-\gamma\}$. Since the root $\gamma$ is connected to $\delta$, we have a connection $\{\zeta_1,\ldots,\zeta_n\} \subset \Lambda \cup \Gamma$ with $n \geq 2$, from $\gamma$ to $\delta$ satisfying:

\medskip

$\zeta_1 = \gamma,$

$\zeta_1 + \zeta_2 \in \Gamma,$

$\hspace{0.5cm} \vdots$

$\zeta_1 + \zeta_2 + \cdots + \zeta_{n-1} \in \Gamma,$

$\zeta_1 + \zeta_2 + \cdots + \zeta_n \in \{\delta,-\delta\}.$

\medskip

\noindent Taking into account $\zeta_1 \in \Gamma_I$ we have that if $\zeta_2 \in \Lambda$ (respectively, $\zeta_2 \in \Gamma$),  the root-multiplicativity and the maximal length of $L$ allow us to  assert that $$0 \neq A_{\zeta_2}L_{\zeta_1} = L_{\zeta_1 + \zeta_2} \hspace{0.4cm} (\hbox{respectively, } 0 \neq [L_{\zeta_1}, L_{\zeta_2}] = L_{\zeta_1 + \zeta_2}).$$ Since $0 \neq L_{\zeta_1} \subset I$, as consequence of Equation \eqref{I-1}, we get in both cases that $$0 \neq L_{\zeta_1+\zeta_2} \subset I.$$

\noindent A similar argument applied to $\zeta_1+\zeta_2 \in \Gamma, \zeta_3 \in \Lambda \cup \Gamma$ and $\zeta_1 + \zeta_2 + \zeta_3 \in \Gamma$ gives us $0 \neq L_{\zeta_1+\zeta_2+\zeta_3} \subset I.$ We can iterate this process with the connection $\{\zeta_1,\ldots,\zeta_n\}$ to get $$0 \neq L_{\zeta_1+\zeta_2+\cdots + \zeta_n} \subset I.$$

\noindent Thus we have shown that
\begin{equation}\label{either-1}
\mbox{for any } \delta \in \Gamma, \mbox{ we have that } 0 \neq L_{\epsilon_{\delta} \delta} \subset I \mbox{ for some } \epsilon_{\delta} \in \{1,-1\}.
\end{equation}

\noindent Since $-\gamma \in \Gamma_I$ then $\{-\zeta_1,\ldots,-\zeta_n\}$ is a connection from $-\gamma$ to $\delta$ satisfying $$-\zeta_1 - \zeta_2 - \cdots -\zeta_n = -\epsilon_{\delta} \delta.$$ By arguing as above we get,
\begin{equation}\label{I-III}
0 \neq L_{-\epsilon_{\delta} \delta} \subset I
\end{equation}
and so $\Gamma_I = \Gamma.$ From the fact $H = \bigl(\sum\limits_{\gamma \in \Lambda \cap \Gamma}A_{\gamma}L_{-\gamma}\bigr) + \bigl(\sum\limits_{\gamma \in \Gamma}[L_{\gamma},L_{-\gamma}]\bigr)$ we also have
\begin{equation}\label{equa-2}
H \subset I.
\end{equation}
From Equations \eqref{I-1}-\eqref{equa-2} we obtain $L \subset I,$ and so $L$ is simple.

\medskip

In the second case, suppose that for any $\gamma \in \Gamma_I$ we have that $-\gamma \notin \Gamma_I.$ Observe that by arguing as in the previous case  we can write
\begin{equation}\label{eq50}
\Gamma = \Gamma_I \dot{\cup} -\Gamma_I
\end{equation}
where $-\Gamma_I := \{-\gamma: \gamma \in \Gamma_I\}.$ Let us denote by $$I' := (\sum\limits_{-\gamma \in -\Gamma_I, \gamma \in \Lambda}A_{\gamma}L_{-\gamma}) \oplus (\bigoplus\limits_{-\gamma \in -\Gamma_I}L_{-\gamma}).$$ Our next aim is to show that $I'$ is an ideal of $L$. Let us prove that $I'$ is a Lie ideal of $L.$ Since $H$ is abelian we have $$[L,I'] = \Bigl[H \oplus (\bigoplus\limits_{\delta \in \Gamma}L_{\delta}), \bigl(\sum\limits_{-\gamma \in -\Gamma_I,\gamma \in \Lambda}A_{\gamma}L_{-\gamma}\bigr) \oplus \bigl(\bigoplus\limits_{-\gamma \in -\Gamma_I}L_{-\gamma}\bigr)\Bigr] \subset$$
\begin{equation}\label{eq40}
(\bigoplus\limits_{-\gamma \in -\Gamma_I}L_{-\gamma}) + \Bigl[\bigoplus\limits_{\delta \in \Gamma}L_{\delta}, \bigl(\sum\limits_{-\gamma \in -\Gamma_I, \gamma \in \Lambda}A_{\gamma}L_{-\gamma}\bigr)\Bigr] + \Bigl[\bigoplus\limits_{\delta \in \Gamma}L_{\delta},\bigl(\bigoplus\limits_{-\gamma \in -\Gamma_I}L_{-\gamma}\bigr)\Bigr].
\end{equation}

Consider the second summand in Equation \eqref{eq40}. If some
$[L_{\delta},A_{\gamma}L_{-\gamma}] \neq 0$ we have that in case $\delta = -\gamma$,  clearly $[L_{-\gamma},A_{\gamma}L_{-\gamma}] \subset L_{-\gamma} \subset I',$ and that in case $\delta = \gamma$, since $I$ is an ideal $-\gamma \notin \Gamma_I$ implies $[L_{-\gamma},A_{-\gamma}L_{\gamma}] = 0,$ we get by maximal length and symmetry of $\Gamma$ that $[L_{\gamma},A_{\gamma}L_{-\gamma}] = 0$. Suppose $\delta \notin \{\gamma, -\gamma\}$. Then by Equation \eqref{fundamental} either $A_{\gamma}[L_{\delta},L_{-\gamma}] \neq 0$ or
$\rho(L_{\delta})(A_{\gamma})L_{-\gamma} \neq 0$ and, by the maximal length of $L$, either $A_{\gamma}[L_{\delta},L_{-\gamma}] = L_{\delta}$ or $\rho(L_{\delta})(A_{\gamma})L_{-\gamma} = L_{\delta}.$ In both cases, since $\gamma \in \Gamma_I$, we have by root-multiplicativity that $L_{-\delta} \subset I,$ that is, $-\delta \in \Gamma_I$. From here $\delta \in -\Gamma_I$ and then $L_{\delta} \subset I'$. Therefore
$$\Bigl[\bigoplus\limits_{\delta \in \Gamma}L_{\delta},\sum\limits_{-\gamma \in -\Gamma_I, \gamma \in \Lambda}A_{\gamma}L_{-\gamma}\Bigr] \subset I'.$$

\noindent Finally, if we consider the third summand in \eqref{eq40} and some $[L_{\delta},L_{-\gamma}] \neq 0$, we have $[L_{\delta}, L_{-\gamma}] = L_{-\gamma + \delta}$. Suppose $\delta \neq \gamma.$ Since $\gamma \in \Gamma_I$, the root-multiplicativity gives us $[L_{\gamma},L_{-\delta}] = L_{\gamma - \delta} \subset I$. Hence $-\gamma +\delta \in -\Gamma_I$ and then $L_{-\gamma + \delta} \subset I'$. Consider $\delta = \gamma,$ in case  $ [L_{\gamma},L_{-\gamma}]  \neq 0$ we have $[L_{\gamma},L_{-\gamma}]  \subset I$ since $\gamma \in \Gamma_I$. Then $L_{-\gamma} = [[L_{\gamma},L_{-\gamma}],L_{-\gamma}] \subset I.$ From here  $\gamma,-\gamma \in \Gamma_I,$ a contradiction with \eqref{eq50}. Thus $[\bigoplus\limits_{\delta \in \Gamma}L_{\delta},\bigoplus\limits_{-\gamma \in -\Gamma_I}L_{-\gamma}] \subset I'$ and $I'$ is a Lie ideal of $L$.

Let us check $AI' \subset I'. $ We have $$AI' = \bigl(A_0 \oplus \bigoplus\limits_{\alpha \in \Lambda}A_{\alpha}\bigr)\Bigl(\bigl(\sum\limits_{-\gamma \in -\Gamma_I, \gamma \in \Lambda}A_{\gamma}L_{-\gamma}\bigr) \oplus \bigl(\bigoplus\limits_{-\gamma \in -\Gamma_I}L_{-\gamma}\bigr)\Bigr) \subset $$
\begin{equation}\label{eq400}
I' + (\bigoplus\limits_{\alpha \in \Lambda}A_{\alpha})\bigl(\sum\limits_{-\gamma \in -\Gamma_I, \gamma \in \Lambda}A_{\gamma}L_{-\gamma}\bigr) + (\bigoplus\limits_{\alpha \in \Lambda}A_{\alpha})\bigl(\bigoplus\limits_{-\gamma \in -\Gamma_I}L_{-\gamma}\bigr).
\end{equation}

Consider the third summand in \eqref{eq400} and suppose that $A_{\alpha}L_{-\gamma} \neq 0$ for certain $\alpha \in \Lambda, -\gamma \in -\Gamma_I.$ In case $\alpha - \gamma \in \Gamma_I,$ we have by the  root-multiplicativity of $(L,A)$ that  $A_{-\alpha}L_{\gamma}\neq 0.$ Now by the maximal length of $L$ and the fact $\gamma \in \Gamma_I$, we get $A_{-\alpha}L_{\gamma} = L_{-\alpha+\gamma} \subset I.$ Therefore $-\alpha + \gamma \in \Gamma_I$ a contradiction. Hence  $\alpha - \gamma \in -\Gamma_I$.

We can argue as above with  the second summand in \eqref{eq400} so as to  conclude that $I'$ is an ideal of the split Lie-Rinehart algebra $(L,A)$.

Now since $[I',I] = 0$ it follows $\sum\limits_{\gamma \in \Gamma}[L_{\gamma},L_{-\gamma}] = 0,$ so by hypothesys must be $$H = \bigl(\sum\limits_{\gamma \in \Gamma_I, -\gamma \in \Lambda}A_{-\gamma}L_{\gamma}\bigr) \oplus \bigl(\sum\limits_{-\gamma \in -\Gamma_I, \gamma \in \Lambda}A_{\gamma}L_{-\gamma}\bigr).$$ Indeed, the sum is direct because if there exists $$0 \neq h \in \bigl(\sum\limits_{\gamma \in \Gamma_I, -\gamma \in \Lambda}A_{-\gamma}L_{\gamma}\bigr) \cap \bigl(\sum\limits_{-\gamma \in -\Gamma_I, \gamma \in \Lambda}A_{\gamma}L_{-\gamma}\bigr),$$  taking into account ${\mathcal Z}(L)=\{0\}$ and $L$ is split, there exists $0 \neq v_{\gamma'} \in L_{\gamma'}, \gamma' \in \Gamma$, such that $[h,v_{\gamma'}] \neq 0$, being then $L_{\gamma'} \subset I \cap I' = 0$, a contradiction. Hence $h \in \mathcal{Z}(L) = \{0\}$ and the sum is direct. Taking into account the above observation and Equation (\ref{eq50}) we have $$L = I \oplus I'.$$

\noindent  Finally, we can proceed with $I$ and $I'$ as we did for $L$ in the first case of the proof to conclude that $I$ and $I'$ are simple ideals, which completes the proof of the theorem.
\end{proof}

\noindent In a similar way to Proposition \ref{teo-3} we can prove the next result.

\begin{proposition}\label{teo-33}
Let $(L,A)$ be a tight split Lie-Rinehart algebra of maximal length, root-multiplicative and all its nonzero weights are connected.  Then either $A$ is simple or $A = J \oplus J'$ where $J$ and $J'$ are simple ideals of $A$.
\end{proposition}

\noindent Finally, we can prove the following theorem.

\begin{theorem}\label{teo100}
Let $(L,A)$ be a tight split Lie-Rinehart algebra  of maximal length, root multiplicative, with symmetric roots and weights systems in such a way that $\Gamma$ have all its nonzero roots connected and $\Lambda$ have all its nonzero weights connected. Then $$\hbox{$L =\bigoplus\limits_{i\in I}L_i$ \hspace{0.4cm} and \hspace{0.4cm} $A = \bigoplus\limits_{j \in J}A_j$}$$ where any $L_i$ is a simple  ideal of $L$  having all of
 its nonzero roots connected  and such that  $[L_i,L_k]=0$ for any $k \in I$ with $i \neq k$;  and any   $A_j$ is  a simple   ideal  of $A$ satisfying  $A_jA_h=0$ for any $h \in J$  such that $j \neq h.$ Furthermore, for any $i \in I$ there exists a unique $\tilde{i} \in J$ such that $$A_{\tilde{i}}L_i \neq 0.$$
  We also have that  any $L_i$ is a split Lie-Rinehart algebra over $A_{\tilde{i}}$.
   %and that $(L,A)$ is the external direct sum of the family of split Lie-Rinehart algebras $\{(L_i,A_{\tilde{i}})\}_{i \in I}$.
\end{theorem}

\begin{proof}
Taking into account Theorem \ref{lema_final}  we can write $$L = \bigoplus\limits_{[\gamma] \in \Gamma/\sim}I_{[\gamma]}$$ as the
direct sum of the family of ideals $I_{[\gamma]}$,  being each $I_{[\gamma]}$ a split  Lie-Rinehart algebra having as roots system $[\gamma]$. Also we can write  $A$ as the direct sum of the  ideals $$A = \bigoplus\limits_{[\alpha] \in \Lambda/\approx}\mathscr{A}_{[\alpha]} $$
%satisfying  $\mathscr{A}_{[\alpha]}\mathscr{A}_{[\psi]} = 0$ when $[\alpha] \neq [\psi]$ and
in  such a way that any $\mathscr{A}_{[\alpha]}$  has as weights system $[\alpha]$, and that for any $I_{[\gamma]}$ there exists a unique $\mathscr{A}_{[\alpha]}$ satisfying $\mathscr{A}_{[\alpha]}I_{[\gamma]} \neq 0$ and  being $(I_{[\gamma]}, \mathscr{A}_{[\alpha]})$  a
split Lie-Rinehart algebra.

In order to apply Proposition \ref{teo-3} and Proposition \ref{teo-33} to each $(I_{[\gamma]}, \mathscr{A}_{[\alpha]})$, we previously have to observe that the root-multiplicativity of $(L,A)$, Proposition \ref{pro-9} and Theorem \ref{teo2}
show that $[\gamma]$  and $[\alpha]$ have, respectively,  all of their elements $\{[\gamma], [\alpha]\}$-connected (that is, connected through connections contained in $[\gamma]$ and $ [\alpha]$. Any of the $(I_{[\gamma]},\mathscr{A}_{[\alpha]})$ is root-multiplicative as consequence of the root-multiplicativity of $(L,A)$. Clearly $(I_{[\gamma]},\mathscr{A}_{[\alpha]})$ is of maximal length and tight, last fact consequence of tightness of $(L,A)$, Proposition \ref{teo-3} and Proposition \ref{teo-33}.  So we can apply Proposition \ref{teo-3} and Proposition \ref{teo-33} to each $(I_{[\gamma]}, \mathscr{A}_{[\alpha]})$ so as to conclude that any $I_{[\gamma]}$ is either simple or the direct sum of simple ideals $I_{[\gamma]} = V \oplus V'$; and that  any $\mathscr{A}_{[\alpha]}$ is either simple or the direct sum of simple ideals $\mathscr{A}_{[\alpha]} = B \oplus B'$. From here, it is clear that by writing $I_i = V \oplus V'$ and $\mathscr{A}_{j} = B \oplus B'$ if $I_i$ or $\mathscr{A}_{j}$ are not, respectively,  simple, then   Theorem \ref{lema_final} allows as to assert that  the resulting decomposition satisfies the assertions of the theorem.
\end{proof}

\end{document}